\documentclass{amsart}
\usepackage{verbatim}
\usepackage{rotating} 
\usepackage{amssymb}
\usepackage{mathrsfs,amsfonts,amsmath,amssymb,epsfig,amscd,xy,amsthm,pb-diagram} 

\newtheorem{theorem}{Theorem}[section]
\newtheorem{proposition}[theorem]{Proposition}

\newtheorem{lemma}[theorem]{Lemma}

\newtheorem{question}[theorem]{Question}

\newtheorem{conjecture}[theorem]{Conjecture}

\theoremstyle{plain}

\theoremstyle{remark}

\newtheorem{remark}[theorem]{Remark}

\newcommand{\F}{{\mathbb F}}
\newcommand{\Q}{{\mathbb Q}}

\newcommand{\Z}{{\mathbb Z}}

\newcommand{\cO}{\mathcal{O}}

\newcommand{\cU}{\mathcal{U}}

\newcommand{\cT}{\mathcal{T}}
\newcommand{\cS}{\mathcal{S}}

\author{Zafer Selcuk Aygin}
\address{Zafer Selcuk Aygin \\
Department of Mathematics and Statistics\\
University of Calgary\\
AB T2N 1N4, Canada}
\email{selcukaygin@gmail.com}

\author{Khoa D.~Nguyen}
\address{
Khoa D.~Nguyen \\
Department of Mathematics and Statistics\\
University of Calgary\\
AB T2N 1N4, Canada
}
\email{dangkhoa.nguyen@ucalgary.ca}

\keywords{Monogenicity, pure cubic fields, function fields, ABC}
\subjclass[2010]{Primary: 11R16, 11R58. Secondary: 11D25}

\begin{document}
	\title{Monogenic pure cubics}
	
	\date{September 2020}
	
	\begin{abstract}
	Let $k\geq 2$ be a square-free integer. We prove that the number 
	of 
	square-free integers
	$m\in [1,N]$ such that $(k,m)=1$ and $\Q(\sqrt[3]{k^2m})$ is 
	monogenic is
	$\gg N^{1/3}$ and $\ll N/(\log N)^{1/3-\epsilon}$ for any $\epsilon>0$. Assuming ABC, the upper bound can be improved to $O(N^{(1/3)+\epsilon})$. Let $F$ be the finite field of order $q$ with $(q,3)=1$ and let $g(t)\in F[t]$ be non-constant square-free. We  prove unconditionally the analogous result that the number of  square-free $h(t)\in F[t]$ such that $\deg(h)\leq N$, $(g,h)=1$ and $F(t,\sqrt[3]{g^2h})$ is monogenic is $\gg q^{N/3}$ and $\ll N^2q^{N/3}$.
	\end{abstract}
	
	\maketitle
	
	\section{introduction}
	A number field $K$ is called monogenic if its ring of integers 
	$\cO_K$ is $\Z[\theta]$ for some $\theta\in\cO_K$. 
	Number fields that are fundamental to the development of 
	algebraic number theory  such as quadratic and cyclotomic 
	fields are all monogenic. Certain questions about monogenic
	number fields (as well as monogenic orders)
	are closely related to the so called discriminant form 
	equations which have been studied extensively by 
	Evertse, Gy\H{o}ry, and other authors. The readers are 
	referred to \cite{EG15_UE,EG16_DE,Ngu17_OM,BN18_SF,Gaa19_DE}
	and the references there for many interesting results including 
	those over positive characteristic fields.

	A pure cubic field
	is a number field of the form $\Q(\sqrt[3]{n})$ where $n>1$ is 
	cube-free. In a certain sense, 
	pure cubic fields are the ``next'' family of number 
	fields 
	to investigate
	after quadratic fields especially from the computational  
	point of view (for example, see \cite{WCS80_CO,WDS83_AR,SS99_VA} 
	of which the third paper treats the function field analogue of 
	pure cubic fields).
	While every
	quadratic field is monogenic, many pure cubics are not and the
	goal of this paper is to study the density of
	monogenic pure cubic fields and its function field analogue.
	For instance, the
	first naive question is whether the set 
	$$\{\text{cube-free $n>1$ such that $\Q(\sqrt[3]{n})$ is monogenic}\}$$
	has zero density. It turns out that the answer is negative thanks 
	to the below theorem of Dedekind. Satisfactory results have been 
	obtained by Bhargava, Shankar, and Wang \cite{BSW_SV} in which they 
	establish the density of monic integer polynomials of degree $n$ 
	having squarefree discriminants and, in a certain sense, the 
	density of monogenic number fields of degree $n$ for any $n>1$. The 
	questions considered in this paper are quite different in nature 
	since we restrict to the 1-parameter family $\Q(\sqrt[3]{n})$ as 
	well as
	the family of polynomials $X^3-n$ none of which have 
	square-free discriminant.

	To see why the above question has a negative answer, we start with 
	the following \cite[p.~35--36]{Mar18_NF}:
		
	\begin{theorem}[Dedekind]\label{thm:DedekindZ}
	Let $n>1$ be a cube-free integer, let $\alpha=\sqrt[3]{n}$, and write $n=k^2m$ where $k$ and $m$ are
	square-free positive integers. We have the following:
	\begin{itemize}
	    \item If $n\not\equiv\pm 1$ mod $9$ then $\{1,\alpha,\alpha^2/k\}$ is
	    an integral basis of $K$.
	    \item If $n\equiv \pm 1$ mod $9$ then $\{1,\alpha,(k^2\pm k^2\alpha+\alpha^2)/(3k)\}$ is an integral basis of $K$.
	\end{itemize}
	\end{theorem}
	
	An immediate consequence is  
	that $\Q(\sqrt[3]{n})$ is monogenic 
	when $n>1$ is square-free and $n\not\equiv \pm 1$ mod $9$. 
	In fact this is the \emph{only} case 
	when we have a positive density result. 
	For  the remaining cases (i.e. $k>1$ or $n\equiv \pm 1$ mod $9$), the conclusion is in stark contrast with the above. As a side
	note, a recent paper of Gassert, Smith, and Stange \cite{GSS19_AF}
	considers the 1-parameter family of quartic fields given by
	$X^4-6X^2-\alpha X-3$ and shows that a positive density of them
	are monogenic.
	
	Throughout this paper, for each square-free positive 
	integer $k$, let:
	$$\cS_k=\{\text{square-free $m>0$: $(m,k)=1$, $k^2m\not\equiv \pm 1$ mod $9$,
	$\Q(\sqrt[3]{k^2m})$ is monogenic}\},$$
	and if $(k,3)=1$ let
	$$\cT_k=\{\text{square-free $m>0$: $(m,k)=1$,
	$k^2m\equiv \pm 1$ mod $9$, 
	$\Q(\sqrt[3]{k^2m})$ is monogenic}\}.$$
	From now on, whenever $\cT_k$ is mentioned, we tacitly assume
	the condition that $(k,3)=1$.
	Our  main results for pure cubic number fields
	are the following:
	
	\begin{theorem}\label{thm:unconditional}
    For every $\epsilon>0$ and square-free integer $k\geq 2$, we have
	$N^{1/3}\ll_k \vert\cS_k\cap [1,N]\vert\ll_{k,\epsilon} N/(\log N)^{1/3-\epsilon}$ as $N\to\infty$.
	For every square-free $k\geq 1$, we have
	$N^{1/3}\ll_k\vert\cT_k\cap [1,N]\vert\ll_{k,\epsilon} N/(\log N)^{1/3-\epsilon}$ as $N\to\infty$. 
	Consequently, the sets
	$S_k$ for $k\geq 2$ and the sets $T_k$ for $k\geq 1$
	have zero density.
	\end{theorem}

	A table of monogenic pure cubic fields with discriminant up to
	$12\cdot 10^6$
	has been computed by Ga\'al-Szab\'o \cite{GS10_AN,Gaa19_DE}
	and it is noted in \cite[p.~111]{Gaa19_DE} that ``the frequency of
	monogenic fields is decreasing''. Our zero density 
	result
	illustrates this observation. Further investigations and computations involving integral bases and monogenicity of higher degree pure number fields
	have been done by Ga\'al-Remete \cite{GR17_IB}.
	 
	Assuming ABC, we can arrive at the much stronger upper bound:
	\begin{theorem}\label{thm:1/3+epsilon}
	Assume that the ABC Conjecture holds. Let $\epsilon>0$ and let $k$ be
	a square-free positive integer.
	We have
	$\vert \cT_k\cap [1,N]\vert=O_{\epsilon,k}(N^{(1/3)+\epsilon})$.
	And if $k\geq 2$, we have
	$\vert \cS_k\cap [1,N]\vert=O_{\epsilon,k}(N^{(1/3)+\epsilon})$.
	\end{theorem}
	
	\begin{remark} 
	From the lower bound in Theorem~\ref{thm:unconditional}, we have 
	that
	the exponent $1/3$ in Theorem~\ref{thm:1/3+epsilon} is best 
	possible. It is not clear if we can replace 
	$N^{(1/3)+\epsilon}$ by some $N^{1/3}f(N)$ where $f(N)$ is
	dominated by $N^{\epsilon}$ for any $\epsilon$.
	\end{remark} 
	
	\begin{remark}
	In principle, we can break $\cT_k$ into 
	$\cT_k^+=\{m\in \cT_k:\ k^2m\equiv 1\bmod 9\}$
	and $\cT_k^-=\{m\in\cT_k:\ k^2m\equiv -1\bmod 9\}$. When choosing the $\pm$ signs appropriately, all results and arguments for $\cT_k$ remain valid
	for each individual $\cT_k^+$ and $\cT_k^-$. 
	\end{remark}

	We now consider the function field setting. For the rest of this section, let $F$ be a finite field
	of order $q$ and characteristic $p\neq 3$. A polynomial $f(t)\in F[t]$ 
	is called square-free (respectively cube-free) if it is not divisible
	by the square (respectively cube) of a 
	\emph{non-constant} element of $F[t]$. Every cube-free $f(t)$ can be written uniquely as $f(t)=g(t)^2h(t)$ in which $g(t),h(t)\in F[t]$
	are square-free and $g(t)$ is monic. We have the analogue of Dedekind's theorem for $F[t]$:
	\begin{theorem}[function field Dedekind]\label{thm:DedekindFt}
	Let $f(t)\in F[t]$ be cube-free, let $\alpha=\sqrt[3]{f}$, $K=F(t,\alpha)$,
	and let $\cO_K$ be the integral closure of $F[t]$ in $K$. Express
	$f(t)=g(t)^2h(t)$ as above. Then 
	$\{1,\alpha,\alpha^2/g\}$ is a basis of $\cO_K$ over $F[t]$.
	\end{theorem}
	\begin{proof}
	The proof is a straightforward adaptation of steps in the proof
	of Theorem~\ref{thm:DedekindZ} given in \cite[p.~35--36]{Mar18_NF}
	\end{proof}
	
	As before, $K=F(t,\sqrt[3]{f})$ is called monogenic if $\cO_K=F[t,\theta]$ for some 
	$\theta\in \cO_K$. For each monic square-free $g(t)\in F[t]$, let
	$$\cU_g=\{\text{square-free $h\in F[t]:$ $(g,h)=1$, $F(t,\sqrt[3]{g^2h})$ is monogenic}\}.$$
	For each positive integer $N$, let $F[t]_{\leq N}$ denote the set of polynomials of degree at most $N$. It is easy to show that there are
	$q^{N+1}-q^N$ square-free polynomials in $F[t]_{\leq N}$. Therefore, if we define the density of a subset $A$ of $F[t]$ to be 
	$$\lim_{N\to\infty}\frac{\vert A\cap F[t]_{\leq N}\vert}{\vert F[t]_{\leq N}\vert}$$
	(assuming the limit exists), then the set $\cU_1$ has density $1-1/q$. As before, this is in stark contrast to the case $\deg(g)>0$:
	\begin{theorem}\label{thm:Ft setting}
	Let $g$ be a non-constant monic square-free polynomial in $F[t]$. 
	We have 
	$$q^{N/3}\ll \vert \cU_g\cap F[t]_{\leq N}\vert \ll N^2q^{N/3}$$
	as $N\to\infty$ where the implied constants depend only on $F$
	and $g$.
	\end{theorem}
	
	\begin{remark} 
	In Theorem~\ref{thm:Ft setting}, an analogous upper bound to
	the number field setting would be $q^{(1/3+\epsilon)N}$. The bound
	$N^2q^{N/3}$ obtained here is much stronger; this is a typical
	phenomenon thanks to the uniformity of various 
	results over function fields.
	\end{remark}
	
	We end this section with a brief discussion on the methods 
	of the proofs. As mentioned above, it is well-known that 
	monogenicity is equivalent to the fact that a certain
	discriminant form equation has a solution
	in $\Z$ (or $\F[t]$ if we are in the function field case). 
	For the questions involving pure cubic fields considered here,
	we end up with an equation of the form $aX^3+bY^3=c$ where 
	$a$ and $c$ are fixed and $b$ varies so that the equation 
	has a solution $(X,Y)$. There are several methods to study
	those Thue equations \cite{EG15_UE,EG16_DE,Gaa19_DE} and we can
	effectively bound the number of solutions or the size of a
	possible solution. However, the question considered here is
	somewhat different: we are estimating how many $b$ for which we 
	have at least one solution. 
	The unconditional upper bound $N/(\log(N))^{1/3}$ in the number 
	field case
	follows from a sieving argument together with a simple instance of 
	the Chebotarev density theorem. The much stronger bound
	$N^{(1/3)+\epsilon}$ in the number field case as well as the
	bound $N^2q^{N/3}$ in the function field case
	follow from the use of ABC together with several combinatorial 
	arguments that might be of independent interest. 
	
	\textbf{Acknowledgments.} We wish to thank Professors Shabnam Akhtari and Istv\'{a}n Ga\'{a}l for helpful comments. Z.~S.~A is partially supported by a PIMS Postdoctoral Fellowship. K.~N. is partially supported by an NSERC Discovery Grant and a CRC tier-2 research stipend.
	
	\section{The number field case} 
	We start with the following:
	\begin{proposition}\label{prop:kX^3-mY^3}
    Let $k$ and $m$ be square-free positive integers. We have:
    \begin{itemize}
        \item [(a)] $m\in \cS_k$ if and only if $k^2m\not\equiv\pm 1$ mod $9$ and the equation $kX^3-mY^3=1$ has a solution $X,Y\in\Z$.
        
        \item [(b)] $m\in\cT_k$ if and only if $k^2m\equiv \pm 1$ mod $9$ and the equation $kX^3-mY^3=9$ has a solution $X,Y\in\Z$.
    \end{itemize}

	\end{proposition}
	\begin{proof}
	For (a), suppose $k^2m\not\equiv \pm 1$ mod $9$,
	let $\alpha=\sqrt[3]{k^2m}$, $K=\Q(\alpha)$,
	and consider the integral basis $\{1,\alpha,\alpha^2/k\}$.
	To find $\theta\in \cO_K$ such that $\cO_K=\Z[\theta]$, 
	it suffices to consider $\theta$ of the form
	$\theta=u\alpha+v(\alpha^2/k)$
	with $u,v\in\Z$. Then we have:
	$$\theta^2=2uvkm+v^2m\alpha+u^2\alpha^2.$$
	We represent $(1,\theta,\theta^2)$
	in terms of the given integral basis and the corresponding matrix has determinant
	$ku^3-mv^2$. Therefore $\cO_K=\Z[\theta]$ if and only if
	the equation $kX^3-mY^3=1$ has a solution $X,Y\in\Z$.
	
	\medskip
	
	The proof of part (b) is similar with some tedious algebraic expressions as follows. Suppose $k^2m\equiv \pm 1$ mod $9$,
	let $\alpha=\sqrt[3]{k^2m}$, $K=\Q(\alpha)$,
	and consider the integral basis $\{1,\alpha,\beta=(k^2\pm k^2\alpha+\alpha^2)/(3k)\}$. As before, consider
	$\theta=u\alpha+v\beta$ with $u,v\in\Z$. Then depending on whether $k^2m\equiv \pm 1$ mod $9$, we have:
	\begin{align*}
	    \theta^2 &=u^2\alpha^2+\frac{2uv}{3k}(k^2\alpha\pm k^2\alpha^2+\alpha^3)+\frac{v^2}{9k^2}(k^2\pm k^2\alpha+\alpha^2)^2\\
	    &=\frac{2uvkm}{3}+\frac{v^2k^2}{9}\pm\frac{2v^2k^2m}{9}+\left(\frac{2uvk}{3}+\frac{v^2m}{9}\pm\frac{2v^2k^2}{9}\right)\alpha\\
	    &\ +\left(u^2\pm\frac{2uvk}{3}+\frac{v^2k^2}{9}+\frac{2v^2}{9}\right)\alpha^2\\
	    &=c_1+c_2\alpha+c_3\beta
	\end{align*}
	where 
	$\displaystyle c_3=3k\left(u^2\pm \frac{2uvk}{3}+\frac{v^2k^2}{9}+\frac{2v^2}{9}\right)$,
	$\displaystyle c_2=\frac{2uvk}{3}+\frac{v^2m}{9}\mp k^2u^2-\frac{2uvk^3}{3}\mp\frac{v^2k^4}{9}$, 
	and the precise value of $c_1$ is not needed for our purpose. We represent $(1,\theta,\theta^2)$ in terms of the given integral basis and the corresponding matrix has determinant:
	$$3ku^3\pm 3k^2u^2v+k^3uv^2-\frac{m\mp k^4}{9}v^3=\frac{1}{9}\left(k(3u\pm kv)^3-mv^3\right).$$
	Therefore if $\cO_K=\Z[\theta]$ then the equation
	$kX^3-mY^3=9$ has a solution $X,Y\in\Z$. Conversely, if
	$(X_0,Y_0)$ is a solution, we can choose $v=Y_0$
	and $u=(X_0\mp kY_0)/3$ and we need to explain why $u\in\Z$. 
	From $k^2m\equiv \pm 1$ mod $9$, we have
	$m\equiv \pm k^4$ mod $9$. Using this and
	the equation $kX_0^3-mY_0^3=9$, we have
	$X_0^3\equiv \pm k^3Y_0^3$ mod $9$. Hence
	$X_0\equiv \pm kY_0$ mod $3$.
	\end{proof}

	The following establish the upper bounds in
	Theorem~\ref{thm:unconditional}:
	\begin{proposition}\label{prop:upper bound}
	    Let $a$ and $b$ be positive integers such that
	    $b/a$ is not the cube of a rational number.
	    As $N\to\infty$, the number of integers $m\in [1,N]$ such that the equation
	    $aX^3-mY^3=b$ has an integer solution
	    is $O_{a,b}(N/(\log N)^{1/3})$.
	\end{proposition}
	\begin{proof}
	Let $L=\Q(\sqrt[3]{b/a})$ and let $L'$ be its Galois closure. Let $S$ be the set of primes 
	$p\nmid 3ab$ such that $b/a$
	is not a cube mod $p$. This means $p$ remains a prime in $L$ and 
	$p\cO_L$ splits completely in $L'$; in other words the 
	Frobenius of $p$ with respect to $L'/\Q$ is the conjugacy class
	of the 2 elements of order 3.
	The Chebotarev density theorem gives that $S$ has Dirichlet as well as natural density
	$1/3$. 
	Put $s(x)=\vert S\cap [1,x]\vert$ so that 
	$s(x)=\displaystyle \frac{\pi(x)}{3}+o(\pi(x))$; put $r(x)=s(x)-\displaystyle\frac{\pi(x)}{3}$. Then partial summation gives:
	$$\sum_{p\in S, p\leq x}\frac{1}{p}=\int_{2^{-}}^{x}\frac{ds(t)}{t}=\frac{1}{3}\int_{2^{-}}^{x}\frac{d\pi(t)}{t}+\int_{2^{-}}^{x}\frac{dr(t)}{t}\sim \frac{1}{3}\log\log x\ \text{as $x\to\infty$}$$
	thanks to the Prime Number Theorem and the fact that $r(t)=o(\pi(t))$. 
	This implies
    \begin{equation}\label{eq:NF prod p in S}
    \prod_{p\in S,p\leq x}\left(1-\frac{1}{p}\right)=(\log x)^{-1/3}e^{o(\log\log x)}.
    \end{equation}
	
	Now observe that if 
	$m\in [1,N]$ is divisible by 
	some $p\in S$ then the equation
	$aX^3-mY^3=b$ cannot have an integer equation
	since $b/a$ is not a cube mod $p$.
	By sieving \cite[Chapter~3.2]{MV06_MN}, 
	the number of $m\in [1,N]$ such that
	$p\nmid m$ for all $p\in S$
	is $O\left(\displaystyle\prod_{p\in S,p\leq N}\left(1-\frac{1}{p}\right)N\right)$ and we use 
	\eqref{eq:NF prod p in S} to finish the proof.
	\end{proof}
	
	\begin{proof}[Proof of Theorem~\ref{thm:unconditional}]
	The upper bound in Theorem~\ref{thm:unconditional} follows from  
	Propositions \ref{prop:kX^3-mY^3} and \ref{prop:upper bound}. 
	For the lower bound, first we consider
	$S_k$ and the 
	equation $kX^3-mY^3=1$. We can always take
	$m=kX_0^3-1$ for $X_0\in [1,(N/k)^{1/3}]$
	so that the above equation
	has a solution $(X_0,1)$. We need that $k^2m\not\equiv\pm 1$ mod $9$ and $m$ is square-free for a positive proportion of such $X_0$. A direct calculation shows that regardless
	of the possibility of $k$ mod $9$, we can always
	find $r\in\{0,\ldots,8\}$
	such that $k^2(kr^3-1)\not\equiv \pm 1$ mod $9$.
	We now choose $X_0$ of the form
	$X_0=9t+r$ for $t\in [1,cN^{1/3}]$ where $c$ is
	a positive constant depending only on $k$.
	By classical results of Hooley \cite{Hoo67_OT,Hoo68_OT} (also see 
	\cite{Gra98_AB} for a more general result assuming ABC), the 
	irreducible
	cubic polynomial $f(t)=k(9t+r)^3-1\in\Z[t]$
	admits square-free values for at least 
	$c'cN^{1/3}$ many $t$ where $c'>0$ depends
	only on $k$ and $r$. The proof of
	$N^{1/3}\ll_k \vert \cT_k\cap [1,N]\vert$
	is completely similar.
	\end{proof}
	
	We will obtain the stronger upper bound $O(N^{1/3+\epsilon})$ 
	assuming the ABC Conjecture:
	\begin{conjecture}\label{conj:ABC}
	Let $\epsilon>0$, then there exists a positive constant $C$ depending only on $\epsilon$ such that the following holds. For all
	relatively prime integers $a,b,c\in \Z$ with $a+b=c$, we have:
	$$\max\{\vert a\vert,\vert b\vert,\vert c\vert\}\leq C\left(\prod_{\text{prime}\ p\mid abc}p\right)^{1+\epsilon}$$
	\end{conjecture}
	
	Theorem~\ref{thm:1/3+epsilon} follows
	from Proposition~\ref{prop:kX^3-mY^3} and the following:
	\begin{proposition}
	Assume Conjecture~\ref{conj:ABC}. Let $a$ and $b$ be positive 
	integers such that $b/a$ is
	not the cube of a rational number and let $\epsilon>0$.
	The number of integers $m$ such that $\vert m\vert\leq N$ and
	the equation $aX^3-mY^3=b$ has an integer solution $(X,Y)$
	is $O_{a,b,\epsilon}(N^{(1/3)+\epsilon})$.
	\end{proposition}
	\begin{proof}
	Let $\delta$ be a small positive number depending on $\epsilon$ that will be specified later. 
	The implicit constants in this proof depends only on $a$, $b$, and $\delta$.
	Except for the finitely many $m$ for which $b/m$ is the cube of an integer, any $(m,X_0,Y_0)$ such that $aX_0^3-mY_0^3=b$, $\vert m\vert\leq N$, and $X_0,Y_0\in\Z$ satisfies $mX_0Y_0\neq 0$. An immediate consequence of ABC gives:
	$$\max\{\vert X_0^3\vert,\vert mY_0^3\vert\}\ll \vert mX_0Y_0\vert^{1+\delta}.$$ 
	From $aX_0^3-mY_0^3=b$, we get $\vert Y_0\vert\ll \vert m^{-1/3}X_0\vert$. Combining with the above, we get:
	$\vert X_0\vert^3\ll \vert m^{2/3}X_0^2\vert^{1+\delta}$.
	Put $\delta'=\displaystyle\frac{2(1+\delta)}{3(1-2\delta)}-\frac{2}{3}$
	so that we have:
	$$\vert X_0\vert \ll m^{(2/3)+\delta'}\ \text{and}\ \vert Y_0\vert\ll m^{(1/3)+\delta'}.$$
	Therefore, in order to estimate the number of $m$, we estimate
	the number of pairs $(X_0,Y_0)$ with $X_0=O(N^{(2/3)+\delta'})$
	and $Y_0=O(N^{(1/3)+\delta'})$ such that
	$\displaystyle\frac{aX_0^3-b}{Y_0^3}$ is an integer in $[-N,N]$.

	Fix such a $Y_0$, we have the obvious bound
	$\vert X_0\vert\ll N^{1/3}\vert Y_0\vert$ and we now study the congurence
	$aX_0^3\equiv b$ mod $Y_0^3$.
	Let $p$ be a prime divisor of $Y_0$ and let $d>0$ such that
	$p^d\parallel Y_0$. If $p\nmid ab$, the equation
	$aX^3\equiv b$ mod $p^{3d}$ has at most 3 solutions in
	$\Z/p^{3d}\Z$ thanks to the structure of $(\Z/p^{3d}\Z)^*$. 
	If $p\mid ab$ and $3d>\max\{v_p(a),v_p(b)\}$, for the above
	congruence equation to have a solution, we must have
	that $v_p(b)-v_p(a)$ is a positive integer divisible by $3$
	and any solution must have the form
	$p^{(v_p(b)-v_p(a))/3}x$ where $x$ satisfies
	$x^3\equiv u$ mod $p^{3d-(v_p(b)-v_p(a))}$
	and $u$ is given by $\displaystyle\frac{b}{a}=p^{v_p(b)-v_p(a)}u$.
	Again, there are at most 3 solutions in this case. 
	In conclusion, there are $O(3^{\omega(Y_0)})$ many solutions
	in $\Z/Y_0^3\Z$ of the equation $aX^3\equiv b$ mod $Y_0^3$; here $\omega(n)$ denotes the number of distinct prime factors of $n$.
	
	Overall, the number of pairs $(X_0,Y_0)$ is at most:
	$$\sum_{Y_0=O(N^{(1/3)+\delta'})}O\left(3^{\omega(Y_0)}\left(\frac{N^{1/3}\vert Y_0\vert}{\vert Y_0^3\vert}+1\right)\right).$$
	This is $O(N^{((1/3)+\delta')(1+\delta')})$ since 
	$3^{\omega(Y_0)}$ is dominated by $\vert Y_0\vert^{\delta'}$. Now choosing
	$\delta$ sufficiently small so that 
	$((1/3)+\delta')(1+\delta')<(1/3)+\epsilon$ and we get the desired conclusion.
	\end{proof}
	
	\section{The function field case}
	Throughout this section, let $F$ be a finite field
	of order $q$ and characteristic $p\neq 3$. A polynomial $f(t)\in F[t]$ 
	is called square-free (respectively cube-free) if it is not divisible
	by the square (respectively cube) of a 
	\emph{non-constant} polynomial in $F[t]$. Every cube-free $f(t)$ can be written uniquely as $f(t)=g(t)^2h(t)$ in which $g(t),h(t)\in F[t]$
	are square-free and $g(t)$ is monic. 	
	We have the function field analogue of Proposition~\ref{prop:kX^3-mY^3} whose proof is completely similar:
	\begin{proposition}\label{prop:gX^3-hY^3}
	Let $f,g,h\in F[t]$ be as above. Then $F(t,\sqrt[3]{f})$ is monogenic if
	and only if there exists
	$X,Y\in F[t]$ such that 
	$gX^3-hY^3\in F^*$.
	\end{proposition}
	
	In function fields, the Mason-Stothers theorem plays a similar role to ABC:
	\begin{theorem}[Mason-Stothers]\label{thm:Mason-Stothers}
	Let $E$ be a field and let $A,B,C\in E[t]$ be
	relatively prime polynomials with $A+B=C$. Suppose that
	at least one of the derivatives $A',B',C'$
	is non-zero then
	$$\max\{\deg(A),\deg(B),\deg(C)\}\leq r(ABC)-1$$
	where $r(ABC)$ denotes the number of distinct roots of $ABC$ in $\bar{E}$. 
	\end{theorem}
	
	 In order to guarantee the condition on derivatives in the above theorem, we need:
	 \begin{lemma}\label{lem:one derivative is nonzero}
	 Let $g(t),h(t)\in F[t]$ be non-constant  square-free polynomials. Suppose there exist $X,Y\in F[t]$ such that $gX^3-hY^3\in F^*$. Then there exist $X_1,Y_1\in F[t]$
	 such that  $gX_1^3-hY_1^3\in F^*$ and 
	     at least one of the derivatives $(gX_1^3)'$ and
	     $(hY_1^3)'$ is non-zero.
	 \end{lemma}
	\begin{proof}
	Write $g=g_1\cdots g_u$ and $h=h_1\cdots h_v$
	where the $g_i$'s and $h_j$'s are irreducible over $F$. Let $n$ be the largest non-negative integer such that both
	$gX^3$ and $hY^3$ are $p^n$-th power of 
	some element of $F[t]$. Write
	$gX^3=\tilde{X}^{p^n}$ and 
	$hY^3=\tilde{Y}^{p^n}$, we have
	that $\tilde{X}-\tilde{Y}\in F^*$ and 
	at least one of the derivatives $\tilde{X}'$
	and $\tilde{Y}'$ is non-zero.
	
	Since $p\neq 3$, from $gX^3=\tilde{X}^{p^n}$
	and $hY^3=\tilde{Y}^{p^n}$ we can express $\tilde{X}$ and 
	$\tilde{Y}$ as:
	$$\tilde{X}=g_1^{b_1}\cdots g_u^{b_u}X_0^{3}\ \text{and}\ \tilde{Y}=h_1^{c_1}\cdots h_v^{c_v}Y_0^3$$
	where the $b_i$'s and $c_j$'s are positive integer, $\gcd(X_0,g_1\cdots g_u)=\gcd(Y_0,h_1\cdots h_v)=1$,
	and $b_ip^n-1\equiv c_jp^n-1\equiv 0$ mod $3$
	for $1\leq i\leq u$ and $1\leq j\leq v$.
	
	Hence the $b_i$'s and $c_j$'s have the same non-zero congruence mod $3$. Depending on whether
	they are $1$ mod $3$ or respectively $2$ mod $3$, we can write 
	$$\tilde{X}=gX_1^3\ \text{and}\ \tilde{Y}=hY_1^3$$
	or respectively
	$$\tilde{X}=g^2X_1^3\ \text{and}\ \tilde{Y}=h^2Y_1^3.$$
	
	We need to rule out
	the second possibility above. Indeed, suppose it happens then
	the Mason-Stothers theorem implies:
	  \begin{align*}
	  \max\{2\deg(g)+3\deg(X_1),2\deg(h)+3\deg(Y_1)\}\leq &\deg(g)+\deg(X_1)+\deg(h)\\
	  & +\deg(Y_1)-1,
	  \end{align*}
	contradiction since the RHS is strictly smaller than the average of the 2 terms in the LHS. This finishes the proof.
	\end{proof}
	
	For $P(t)\in F[t]\setminus\{0\}$, let $\omega(P)$ denote the number of
	distinct monic irreducible factors of $P$. As before, we also need an upper bound for
	$\omega(P)$:
	\begin{lemma}
	For every $P(t)\in F[t]$ of degree $N\geq 2$, we have
	$\omega(P)\ll_q N/\log N$.
	\end{lemma}
	\begin{proof}
	   All the implicit constants in this proof depend only on $q$.
	   For every positive interger $k$, let $d_k$ be the degree of the product of all monic irreducible polynomials
	   of degree at most $k$. Since there are 
	   $(q^n+O(q^{n/2}))/n$ monic irreducible polynomials of degree $n$, we have
	   $$d_k=\sum_{n=1}^k (q^{n}+O(q^{n/2}))= \frac{q^{k+1}}{q-1}+O(q^{k/2}).$$
	
	   Now we choose the smallest $k$ such that $N\leq d_k$. This implies that $\omega(P)$ is at most the number of
	   monic irreducible polynomials of degree at most $k$: 
	   $$\omega(P)\leq \sum_{n=1}^k\frac{q^n+O(q^{n/2})}{n}=O(q^k/k).$$
	   From the above formula for $d_k$ and the choice of $k$, we have that $q^k\ll N$ and $k\ll \log N$;
	   this finishes the proof.
	\end{proof}
	
	It turns out that we will need an estimate for
	$\sum 3^{\omega(P)}$ where $\deg(P)\leq N$. Using the above  bound
	$N/\log N$ for each individual $\omega(P)$
	would yield $O(q^{N+O(N/\log N)})$
	for the above sum which would not be good enough for our purpose.
	Instead, we have:
	\begin{lemma}
	$\displaystyle\sum_{\deg (P)\leq N} 3^{\omega(P)}=O(N^2q^N)$.
	\end{lemma}
	\begin{proof}
	   Let $s_N$ be $\sum 3^{\omega(P)}$ where $P$ ranges over all $\emph{monic}$ polynomials of degree equal to $N$. 
	   It suffices to show $s_N=O(N^2q^N)$. The generating series $\displaystyle\sum_{n}s_nT^n$ has the Euler product:
	   $$\prod_{P}(1+3T^{\deg(P)}+3T^{2\deg(P)}+\ldots)=\prod_{P}\frac{1+2T^{\deg(P)}}{1-T^{\deg(P)}},$$
	   where $P$ ranges over all the monic irreducible polynomials over $F$. The denominator is simply the zeta-function
	   $1/(1-qT)$ while the coefficients of the numerator are bounded above
	   by the coefficients of 
	   $\displaystyle\prod_{P}(1+t^{\deg(P)}+t^{2\deg(P)}+\ldots)^2=\frac{1}{(1-qT)^2}$. Therefore the $s_N$'s are bounded
	   above by the coefficients of $1/(1-qT)^3$ and this finishes the proof.
	\end{proof}
	
	\begin{proof}[Proof of Theorem~\ref{thm:Ft setting}]
	    For the lower bound, we simply study the equation
	    $gX^3-hY^3=1$. Either by adapting the arguments in 
	    \cite{Hoo67_OT,Hoo68_OT,Gra98_AB} or using 
	    a general result of Poonen \cite[Theorem~3.4]{Poo03_SV}
	    which is valid for a multivariable polynomial, 
	     we have that
	    for a positive proportion of the polynomials
	    $X\in F[t]$ with $\deg(X)\leq (N-\deg(g))/3$, we have
	    that $gX^3-1$ is squarefree; we now simply take 
	    $Y=1$ and $h=gX^3-1$ for those $X$'s. 
	    This proves the lower bound.
	    
	    For the upper bound, we prove that for an arbitrary
	    $\alpha\in F^*$, there are $O(q^{N/3})$ many $h$ of degree at most $N$ such that the equation $gX^3-hY^3=\alpha$ has a solution $X,Y\in F[t]$; since
	    $\deg(g)>0$ we must have that $Y\neq 0$.
	    By Lemma~\ref{lem:one derivative is nonzero}, we may assume that at least one of the derivatives $(gX^3)'$ and $(hY^3)'$ is non-zero.
	    The Mason-Stothers theorem yields:
	    $$\deg(h)+3\deg(Y)\leq \deg(g)+\deg(X)+\deg(h)+\deg(Y)-1,\ \text{and}$$
	    $$\deg(g)+3\deg(X)\leq \deg(g)+\deg(X)+\deg(h)+\deg(Y)-1.$$
	    The first inequality gives $\deg(Y)\leq \deg(X)/2+O(1)$
	    then we use this and the second inequality to obtain
	    $\deg(X)\leq \frac{2}{3}\deg(h)+O(1)$. Then it follows that $\deg(Y)\leq \frac{N}{3}+O(1)$.
	
	   We now count the number of pairs $(X,Y)$ with $Y\neq 0$ such that 
	   $\deg(Y)\leq \frac{N}{3}+O(1)$ and 
	   $\frac{gX^3-\alpha}{Y^3}$ is a polynomial in $F[t]$ of degree at most $N$.
	   Hence $\deg(X)\leq \deg(Y)+(N/3)-\deg(g)$. Arguing as before, for each prime power factor $P^n$ of $Y$, the congruence equation
	   $gX^3-\alpha=0$ mod $P^{3n}$ has 
	   at most 3 solutions mod $P^{3n}$. 
	   Therefore by the Chinese Remainder Theorem, the congruence equation 
	   $gX^3-\alpha=0$ mod $Y^3$ has at most $3^{\omega(Y)}$ solutions mod $Y^3$.
	   Therefore once $Y$ is fixed, there are at most
	   $$3^{\omega(Y)}\left(q^{\deg(Y)+(N/3)-\deg(g)-3\deg(Y)+1}+1\right)$$
	   possibilities for $X$. Hence the number of pairs $(X,Y)$
	   is at most:
	   \begin{align*}
	   & \sum_{k=0}^{(N/3)+O(1)} \sum_{Y:\ \deg(Y)=k}3^{\omega(Y)}\left(q^{(N/3)-2\deg(Y)}
	   +1\right)\\
	   & \ll \sum_{k=2}^{(N/3)+O(1)}(q-1)q^k3^{O(k/\log k)}q^{(N/3)-2k}+
	   \sum_{k=0}^{(N/3)+O(1)} \sum_{Y:\ \deg(Y)=k}3^{\omega(Y)}.\\
	   \end{align*}
	   The first term is $O(q^{N/3})$ since $\sum_{k=0}^{\infty} 3^{O(k/\log k)}q^{-k}<\infty$ while the second term is $O(N^2q^{N/3})$ thanks to the previous lemma and this finishes the proof.
	\end{proof}
	
	\section{Further Questions}
	Thanks to the lower bounds in our results, we know that the ``main terms'' $N^{1/3}$ and $q^{N/3}$ in the upper bounds are optimal. However, it seems possible that the ``extra factors'' $N^{\epsilon}$
	in the number field case
	and $N^2$ in the function field case can be improved. This motivates:
	\begin{question}
		\begin{itemize}
			\item [(a)] In Theorem~\ref{thm:1/3+epsilon}, can one replace the bound $O(N^{1/3+\epsilon})$ by $O(N^{1/3}f(N))$ where 
			$f(N)$ is dominated by any $N^{\epsilon}$?
			
			\smallskip
			
			\item [(b)] In Theorem~\ref{thm:Ft setting}, can one improve the bound $O(N^2q^{N/3})$? Could this upper bound even be $O(q^{N/3})$?
			
			\smallskip
			
			\item [(c)] In the number field case, can one obtain an unconditional power-saving bound $O(N^c)$ where $c<1$? 		
		\end{itemize}
	\end{question}

	\bibliographystyle{amsalpha}
	\bibliography{MPC} 	

\end{document}